\documentclass{amsart}
\usepackage{amssymb, amsmath}
\usepackage{stmaryrd}
\usepackage{xcolor}

\title{Minimally critical endomorphisms of $\PP^N$}
\author{Patrick Ingram}
\email{pingram@yorku.ca}
\address{York University, Toronto, Canada}
\date{\today}

\renewcommand{\epsilon}{\varepsilon}

\newcommand{\PP}{\mathbb{P}}
\newcommand{\ZZ}{\mathbb{Z}}
\newcommand{\CC}{\mathbb{C}}

\newcommand{\QQ}{\mathbb{Q}}
\renewcommand{\AA}{\mathbb{A}}
\newcommand{\Res}{\operatorname{Res}}

\newcommand{\ord}{\operatorname{ord}}
\newcommand{\GL}{\operatorname{GL}}
\newcommand{\PGL}{\operatorname{PGL}}

\newcommand{\Hom}{{\operatorname{Hom}}}
\newcommand{\MinCrit}{{\operatorname{MinCrit}}}

\newtheorem{theorem}{Theorem}
\newtheorem{lemma}[theorem]{Lemma}
\newtheorem{corollary}[theorem]{Corollary}
\newtheorem{prop}[theorem]{Proposition}

\theoremstyle{definition}

\begin{document}
\begin{abstract}
We study the dynamics of the map $f:\PP^N\to \PP^N$ defined by $f(\mathbf{X})=A\mathbf{X}^d$, 
for $A\in\PGL_{N+1}(\overline{\QQ})$ and $d\geq 2$. When $d>N^2+N+1$, we show that the critical height of such a morphism is comparable to its height in moduli space, confirming a case of a natural generalization of a conjecture of Silverman.
\end{abstract}

\maketitle

A non-linear rational function $f:\PP^1_\CC\to \PP^1_\CC$ has at least two critical points, and so if one is interested in exploring the dynamics of rational functions in one variable, a natural place to start is the family of functions with two critical points and no more; a lovely study of these~\emph{bicritical} functions was carried out by Milnor~\cite{milnor}.

Over $\PP^N$ one might consider a similar minimality condition on ramification, 
restricting attention to \emph{minimally critical} morphisms, that is, those whose critical locus is supported on $N+1$ hyperplanes. It is easy to see from the degree of the ramification divisor that this is the least number, and that such hyperplanes must meet properly. After a change of variables, each such endomorphism can be written as
\begin{equation}\label{eq:Aform}f_A(\mathbf{X})=A\mathbf{X}^d,\end{equation}
for some $A\in \PGL_{N+1}(\CC)$, where $\mathbf{X}^d$ is the entrywise $d$th power of $\mathbf{X}$.

The main result of this paper establishes a version of Silverman's critical height conjecture in the case of minimally critical morphisms of sufficiently high degree. To make this explicit, let $\hat{h}_{\mathrm{crit}}(f)$ be the \emph{critical height} of $f:\PP^N_{\overline{\QQ}}\to\PP^N_{\overline{\QQ}}$, defined as the canonical height (in the sense of Zhang~\cite{zhang}) of the ramification divisor. Generalizing a conjecture of Silverman~\cite[Conjecture~6.29, p.~101]{barbados} to $\PP^N$, it is proposed in~\cite{pnheights} that for any ample Weil height $h_{\mathsf{M}_d^N}$ on the moduli space $\mathsf{M}_d^N$ of endomorphisms of $\PP^N$ with degree $d\geq 2$, there ought to be a proper Zariski-closed subset of $\mathsf{M}_d^N$ away from which $h_{\mathsf{M}_d^N}$ is comparable to the critical height $\hat{h}_{\mathrm{crit}}$ on $\mathsf{M}_d^N$. When $d$ is large, our main result establishes this conjecture on the locus $\operatorname{MinCrit}_d^N\subseteq \mathsf{M}_d^N$ of minimally critical morphisms (in this case with no exceptional subvariety).

\begin{theorem}\label{th:main}
Let $N\geq 1$ and $d> N^2+N+1$. Then for any ample Weil height $h_{\mathsf{M}_d^N}$ on $\mathsf{M}_d^N$  we have
	\[h_{\mathsf{M}_d^N}(f)\asymp \hat{h}_{\mathrm{crit}}(f),\]
	 for $f\in \operatorname{MinCrit}_d^N$,
	with implied constants depending only on $N$ and $d$.
	\end{theorem}
	
Theorem~\ref{th:main} follows fairly quickly from the following, more explicit statement. 
\begin{theorem}\label{th:exp}
For any minimally critical map $f:\PP^N\to\PP^N$ of degree $d>N^2+N+1$ defined over $\overline{\QQ}$, there is an $A\in \PGL_{N+1}(\overline{\QQ})$ such that $f_A$ as in~\eqref{eq:Aform} is conjugate to $f$ and
\[\frac{1}{N(N+1)}\hat{h}_{\mathrm{crit}}(f)-C_1\leq h_{\PGL_{N+1}}(A)\leq\left(\frac{dN}{d-(N^2+N+1)}\right)\hat{h}_{\mathrm{crit}}(f) +C_2\]
for some explicit constants $C_1$ and $C_2$ depending just on $N$ and $d$.	
\end{theorem}

A morphism $f:\PP^N\to\PP^N$ is \emph{post-critically finite (PCF)} if and only if the subset
\[\bigcup_{k\geq 1}f^k(C_f)\subseteq \PP^N\]
is algebraic, where $C_f$ is the ramification divisor of $f$. Equivalently, $f$ is PCF if and only if there exist $k\geq 0$ and $m\geq 1$ so that the support of $f^{k+m}(C_f)$ is contained in that of $f^k(C_f)$. The critical height vanishes on PCF maps, and so results such as Theorem~\ref{th:main} allow us to say something about these examples.  Note, in fact, that most of the known examples of PCF endomorphisms of $\PP^N$ which do not in some way derive either from algebraic groups or from endomorphisms of $\PP^1$ are of the form~\eqref{eq:Aform}. Indeed, in the case $N=d=2$, the matrices
\[A=\begin{pmatrix}
 1 & -2 & 0 \\ 1 & 0 & 0 \\ 1 & 0 & -2	
 \end{pmatrix}, \begin{pmatrix}
 1 & -2 & 0 \\ 1 & 0 & -2	\\  1 & 0 & 0 
 \end{pmatrix}, \begin{pmatrix}
 1 & -1 & 1 \\ 1 & 1 & -1	\\  -1 & 1 & 1 
 \end{pmatrix},  \begin{pmatrix}
 1 & -1 & 0 \\ 1 & 0 & -1	\\  1 & 0 & 0 
 \end{pmatrix}
\]
all define PCF maps, the first two being examples due to Forn\ae ss and Sibony~\cite{fs}, the next to Dupont~\cite{dupont} (a Latt\`{e}s example, along with two more obtained by permuting rows in this matrix), and the last studied by Belk and Koch~\cite{bk}.

So far we have discussed results over number fields, but the machinery which proves Theorem~\ref{th:main} works in most function fields as well.
One consequence  is a rigidity result for PCF maps of the form~\eqref{eq:Aform}. Recall that an algebraic family of endomorphisms of $\PP^N$ is one whose coefficients lie in the function field of some algebraic variety, and an \emph{isotrivial} family is one whose coefficients are constant, maybe after a finite extension of the function field and a change of variables. It is a result of Thurston~\cite[Theorem~6.11, p.~93]{barbados} that any algebraic family of PCF maps on $\PP^1_{\CC}$ is either isotrivial, or a family of Latt\`{e}s examples. Little is known in this direction for endomorphism of $\PP^N$, with $N>1$, but our next theorem addresses this for minimally critical maps of sufficiently high degree. Similar results for other families have previously been established by Gauthier and Vigny~\cite{gv}.
\begin{theorem}\label{th:rigid}
Let $k$ be an algebraically closed field of characteristic either 0 or  $p>d$, and let $d>N^2+N+1$. Then any algebraic family of degree-$d$ minimally critical PCF maps defined over $k$ is isotrivial.
\end{theorem}

One can use essentially the same argument as that which proves Theorem~\ref{th:rigid} to prove that any minimally critical PCF map of degree $d>N^2+N+1$ defined over $\CC$ is conjugate to one of the form~\eqref{eq:Aform} for some 
$A\in\PGL_{N+1}(\overline{\QQ})$. In other words, considering the PCF examples with algebraic coefficients is considering all of them.

Theorem~\ref{th:exp} has the following finiteness result as a consequence, analogous to one of the main results of~\cite{bijl}. 
 \begin{corollary}\label{th:heightbound}
For each $D\geq 1$ and $d>N^2+N+1$, there is a finite and effectively computable set $S_D\subseteq \PGL_{N+1}(\overline{\QQ})$ such that every minimally critical PCF map $f$ of degree $d$ defined over a number field of degree at most $D$  is conjugate to $A\mathbf{X}^d$ for some $A\in S_D$.
\end{corollary}
Note that the only non-standard step in concluding Corollary~\ref{th:heightbound} from  Theorem~\ref{th:exp} is to confirm that the morphism $f_A$ mentioned therein is defined over an extension of the field of definition of $f$ of bounded degree.

In Section~\ref{sec:examples}, we write down examples of minimally critical PCF maps in every degree and dimension. In Section~\ref{sec:normal} we describe something of a normal form for minimally critical maps, working over an algebraically closed field (or even separably closed if $d$ is not divisible by the characteristic). In Sections~\ref{sec:local} and~\ref{sec:lyap} we work over local fields, providing some estimates on local heights and Lyapunov exponents, and in Section~\ref{sec:global} we assemble these bounds to prove the main results.


\section{Examples}\label{sec:examples}

Before continuing, we exhibit minimally critical PCF endomorphisms of $\PP^N$ for all $N$ and of every degree. These examples are a straightforward generalization of those due to Forn\ae ss and Sibony, and Belk and Koch, presented in the introduction. To shed some light on the combinatorics of these examples, we compute the least $k$ and $m$ so that $f^{k+m}(H_i)=f^k(H_i)$ for each coordinate hyperplane $H_i$.
In general, we say that a divisor $D$ on $\PP^N$ has \emph{type} $(k, m)$ relative to $f$ if and only if
\[f^{k+m}(\Gamma)\subseteq f^k(\Gamma)\]
for every irreducible component of $\Gamma$ of the support of $D$, and if $k\geq 0$ and $m\geq 1$ are the smallest such values. We will say that $f$ has \emph{critical type} $(k, m)$ if $C_f$ has type $(k, m)$ relative to $f$. Note that this is different from saying that $k$ and $m$ are minimal with $f^{k+m}(C_f)\subseteq f^k(C_f)$, since different components of $C_f$ might be in the same orbit.   

 In general it is possible to ``fake-up'' PCF examples in several variables from PCF endomorphisms of $\PP^1$, which are in good supply. For instance, if $f$ is a PCF rational function of one variable, then the diagonal action of $f$ is an endomorphism of $(\PP^1)^N$ which commutes with coordinate permutations, and hence induces an endomorphism of $(\PP^1)^N/S_N\cong \PP^N$. This endomorphism will be PCF and have degree $\deg(f)$ (see, e.g., \cite{ds}).
 
 Specifically, for $P=[\alpha:\beta]\in \PP^1$, define a hyperplane $H_P$ in $\PP^N$ by
 \[H_P:\sum_{i=0}^N(-1)^i\alpha^i\beta^{N-i}X_i=0.\]
If $f:\PP^1\to \PP^1$, and $F:\PP^N\to\PP^N$ is the $N$th symmetric power of $f$, then it is straightforward to check that $F(H_P)=H_{f(P)}$. If $\Delta$ is the divisor on $\PP^N$ over which two coordinates of $(\PP^1)^N$ coincide, which is irreducible of degree $2N-2$, then $F^*\Delta=\Delta+2E$, where $E$ is the image of the locus in $(\PP^1)^N$ on which two coordinates are distinct $f$-preimages of the same point. One calculates that the ramification divisor of $F$ is exactly
\[C_F=E + \sum (e_f(P)-1)H_P=E+H_{C_f},\]
if we extend the association $P\mapsto H_P$ linearly to divisors. Since $E\neq \Delta$, but $F(E)=\Delta=F(\Delta)$, we see that whenever $f$ has critical type $(k, m)$, the symmetric power $F$ has critical type $(\max\{1, k\}, m)$. The first type of example constructed below has periodic critical locus (that is, type $(0, m)$), and hence cannot be obtained in this way from an endomorphism of $\PP^1$. The second class of examples has preperiodic critical locus, but  the branch locus consists entirely of hyperplanes, and hence they also cannot be conjugate to symmetric powers (whose branch loci contain the irreducible hypersurface $\Delta$ of degree $2N-2$).

\subsection{Examples in the style of Belk and Koch}
Let $\sigma$ be any permutation of $\{0, ..., N\}$. We define a morphism $f=[f_0:\cdots :f_N]$ by
\[f_{\sigma(i)} = \begin{cases}
 	X_0^d & \text{if }i=0 \\
 	X_0^d-X_i^d & \text{if }i\neq 0.
 \end{cases}
\]
Now, let $H_i$ be the hyperplane defined by $X_i=0$ and $H_{i, j}=H_{j, i}$ by $X_i=X_j$. It is easy to check that
\[f(H_i)=\begin{cases}
H_{\sigma(0)} & \text{if }i=0\\
H_{\sigma(0), \sigma(i)}	 & \text{if }i\neq0,
\end{cases}
\]
while
\[f(H_{i, j})=\begin{cases}
H_{\sigma(j)} & \text{if }i=0\\
H_{\sigma(i), \sigma(j)}	 & \text{if }i, j\neq 0,
\end{cases}
\]
where here and throughout we use $f(D)$ to denote the image of $D$, which is the support of the push-forward of $D$.
In particular, this finite collection of hyperplanes is permuted in some fashion by $f$, and since it contains the support of the critical divisor (the coordinate hyperplanes), $f$ is PCF. The Bell-Koch example in the introduction corresponds to the cyclic permutation $(021)\in S_3$.

In order to formulate the next two results, we recall the \emph{Landau function} $g(N)$, defined as the largest order of an element of $S_N$, which satisfies
 \[g(N)=e^{(1+o(1))\sqrt{N\log N}}.\]

\begin{prop}
There exist minimally critical PCF endomorphisms of $\PP^N$ of every degree with critical type $(0, g(N+2))$, where $g$
is the Landau function.
\end{prop}

\begin{proof}
We compute the combinatorics of the orbit of each coordinate hyperplane for examples as constructed above. In general, let $\sigma\in S_{N+1}$ and let $0$ have period $m\neq 1$ under $\sigma$. First, note that
\[H_0\to H_{\sigma(0)}\to H_{\sigma^2(0), \sigma(0)}\to\cdots\to H_{0, \sigma^{m-1}(0)}\to H_0, \]
and so $H_0$ and $H_{\sigma(0)}$ have type $(0, m+1)$.

If $i=\sigma^k(0)$ with $k\neq 0, 1, \frac{m+1}{2}$ (including the case where $m$ is even), we have (under the action of $f$)
\begin{multline*}
H_{i}=H_{\sigma^k(0)}\to H_{\sigma(0), \sigma^{k+1}(0)}\to\cdots\to H_{\sigma^{m-k-1}(0), \sigma^{m-1}(0)} \to H_{0, \sigma^{m-k}(0)}\\ \to H_{\sigma^{m-k+1}(0)} \to H_{\sigma(0), \sigma^{m-k+2}(0)}\to\cdots \to H_{0, \sigma^{k-1}(0)}\to H_{\sigma^k(0)}=H_i,	
\end{multline*}
and so $H_i$ again has type $(0, m+1)$. Finally, if $k=\frac{m+1}{2}$ then the above still holds, but we have $m-k+1=k$, and so we have run through the cycle twice. This hyperplane has type $(0, \frac{m+1}{2})$.

Now suppose that $i\in \{0, ..., N\}$ is not in the orbit of $0$, and has period $k$ under $\sigma$. Then we have
\begin{multline*}
H_{i}\to H_{\sigma(0), \sigma(i)}\to\cdots \to H_{\sigma^{m-1}(0), \sigma^{m-1}(i)}\to H_{0, \sigma^m(i)}\\ \to H_{\sigma^{m+1}(i)} \to \cdots \to H_{\sigma^{2(m+1)}(i)}\to\cdots H_{\sigma^{j(m+1)}(i)}=H_i,
\end{multline*}
for the first time when $j=k/\gcd(k, m+1)$, so $H_i$ has type $(0, \operatorname{lcm}(k, m+1))$.

Define an element $\sigma'\in S_{N+2}$ with the same cycle structure of $\sigma$, but with the cycle containing $0$ increased by one element. Then if $\sigma'$ has order $\ell$, $f$ as constructed above has critical type $(0, \ell)$. We may now simply select $\sigma'\in S_{N+2}$ of maximal order, with the constraint that $0$ is in a cycle of length at least 3. 
\end{proof}

\subsection{Examples in the style of Forn\ae ss and Sibony}
Similarly, with $d\geq 2$ and $\zeta\neq 1$ a $d$th root of unity, define another morphism $f:\PP^N\to\PP^N$ by
\[f_{\sigma(i)} = \begin{cases}
 	X_0^d & \text{if }i=0 \\
 	X_0^d+(\zeta-1)X_i^d & \text{if }i\neq 0.
 \end{cases}
\]
Here, let $H_i$ and $H_{i, j}$ be as before, and let $H_{i, j}^{\pm}$ be defined by $X_i=\zeta^{\pm 1} X_j$, noting that $H_{i, j}^{\pm}=H_{j, i}^{\mp}$. We then check that
\[f(H_i)=\begin{cases}
H_{\sigma(0)} & \text{if }i=0\\
H_{\sigma(0), \sigma(i)}	 & \text{if }i\neq0,
\end{cases}\]
while
\[f(H_{i, j}^{\pm})=f(H_{i, j})=\begin{cases}
H_{\sigma(0), \sigma(j)}^{+} & \text{if }i=0\\
H_{\sigma(i), \sigma(0)}^{-} & \text{if }j=0\\
H_{\sigma(i), \sigma(j)}	 & \text{if }i, j\neq 0.
\end{cases}
\]
Again we have a finite collection of hyperplanes which is closed under the action of $f$, and contains the support of the ramification divisor, and so $f$ is PCF. Indeed, the first example is essentially the second with $\zeta=0$ (which is not a root of unity, of course). 

\begin{prop}
There exist minimally critical PCF endomorphisms of $\PP^N$ of every degree with critical type $(3, g(N+1))$.
\end{prop} 

\begin{proof}
	We employ the construction immediately preceding the proposition. Let $\sigma\in S_{N+1}$, and suppose that $0$ has order $m\neq 1$ under $\sigma$. Let 
	\[L_{i, j} =L_{j, i}=\begin{cases}
H_{i, j} & \text{if }i, j\neq \sigma(0)\\
H^+_{i, j}=H^{-}_{j, i} & \text{if }i=\sigma(0).
\end{cases}
\]
Then from the above, we see that \[f(L_{i, j})=f(H_{i, j})=L_{\sigma(i), \sigma(j)},\] for any $i, j$, while $L_{i, j}=L_{i', j'}$ implies $\{i, j\}=\{i', j'\}$. If $e_i$ is the period of $i$ under $\sigma$, we therefore have $L_{i, j}$ of period $\operatorname{lcm}(e_i, e_j)$.

For $i\neq 0$, we have $f^m(H_i)=L_{0, i}$, and so $H_i$ has eventual period that of $L_{0, i}$. But $H_i\neq L_{j,k}$ for any $j, k$, and similarly $f(H_i)=H_{\sigma(0), \sigma(i)}$ is not of the form $L_{j, k}$. Since \[f^2(H_i)=f(H_{\sigma(0), \sigma(k)})=L_{\sigma^2(0), \sigma^2(i)},\] we have $H_i$ of tail length exactly 2.

Similarly, $f^m(H_0)=L_{0,\sigma^{m-1}(0)}$, and so $H_0$ has eventual period equal to that of $L_{0,\sigma^{m-1}(0)}$. On the other hand, since $\sigma(0)\neq 0$ we have just computed that $H_{\sigma(0)}$ has tail length $2$, and so $H_0$ has tail length 3.

 Thus the type of $C_f=\sum H_i$ will be \[(\max\{2, 3\}, \operatorname{lcm}(e_0, e_1, ..., e_N)),\] and note that the second coordinate is the order of $\sigma$. The final claim follows by taking $\sigma$ to be any element of $S_{N+1}$ of maximal order (see, e.g., \cite{miller}).
\end{proof}

Note that Dupont~\cite{dupont} shows that one of the examples of Forn\ae ss and Sibony is not a Latt\`{e}s example. It would be interesting to know whether it is possible to apply the same sort of analysis to show that some subset of the examples constructed above are non-Latt\`{e}s.


\section{Normalized representatives}\label{sec:normal}

Let $k$ be an algebraically closed field of characteristic $0$ or $p$, and fix $d\geq 2$ not divisible by $p$. We will briefly describe the space of minimally critical endomorphisms over $k$.

As in~\cite[Chapter~1]{barbados}, let $\Hom_d^N$ be the space of endomorphisms of $\PP^N$ of degree $d$, parametrized by their coefficients. Viewing the coefficients as a projective $M$-tuple, with $M=(\binom{N+d}{d}(N+1)-1)$, we see that $\Hom_d^N\subseteq \PP^M$ is the complement of the vanishing of the Macaulay resultant. The group $\PGL_{N+1}$ of automorphisms of $\PP^N$ acts on $\Hom_d^N$ by conjugation, $f^B=B^{-1}fB$, and the quotient of $\Hom_d^N$ by this action is an affine variety $\mathsf{M}_d^N$ \cite[Theorem~2.24, p.~21]{barbados}.

The association of the map $\mathbf{X}\mapsto A\mathbf{X}^d$ to the matrix $A$ gives a morphism $\PGL_{N+1}\to \Hom_d^N$,  identifying $\PGL_{N+1}$ with an intersection of  coordinate hyperplanes. To see that this is a morphism, note that the extension to the ambient projective space $\PP^{(N+1)^2-1}\supseteq \PGL_{N+1}$ gives an embedding of projective spaces, while for any choice of representative $A\in\GL_{N+1}$ of a point in $\PGL_{N+1}$, we have
\[\Res(A\mathbf{X}^d)=\det(A)^{d^{N}},\]
by \cite[Theorem~3.13, p.~399]{lang} and  \cite[Corollary~3.14, p.~400]{lang}, so this embedding identifies the complement of $\PGL_{N+1}$ with the complement of the resultant locus in its image.

We write $\MinCrit_d^N\subseteq \mathsf{M}_d^N$ for the collection of conjugacy classes of morphisms ramified along $N+1$ hyperplanes.
\begin{lemma}\label{lem:mincritdes}
The locus $\MinCrit_d^N\subseteq \mathsf{M}_d^N$ is precisely the image of $\PGL_{N+1}$ under the composition of the above-described map $\PGL_{N+1}\to \Hom_d^N$ with the quotient map $\Hom_d^N\to \mathsf{M}_d^N$.
\end{lemma}

\begin{proof}
Of course, we can compute the ramification locus of a map of the form $A\mathbf{X}^d$, and confirm that it is indeed of the sort described, so the image of $\PGL_{N+1}$ in $\mathsf{M}_d^N$ is certainly contained in $\MinCrit_d^N$.

On the other hand, suppose that $f:\PP^N\to\PP^N$ of degree $d\geq 2$ is ramified along $N+1$ hyperplanes. Since the ramification index along each hyperplane is at most $d-1$, and since the ramification divisor has degree $(N+1)(d-1)$, we see that $f$ is totally ramified along each of these hyperplanes. Suppose these hyperplanes meet improperly. This cannot happen if $N=1$, since the (in this case 2) hyperplanes would then be equal, and so we must have $N\geq 2$. Post-composing with a linear transformation does not change the ramification locus, and so we may assume, without loss of generality, that one of the ramified hyperplanes $H$ is also fixed. But then one can check that the ramification of the restricted map $f|_H:H\to H$ is supported on the restriction of the remaining hyperplanes, all of which are totally ramified. Since $H\cong \PP^{N-1}$, we now have a minimally critical example in one lower dimension with improperly intersecting ramified hyperplanes. By induction, then, the original example cannot occur.

Now that we know that the hyperplanes meet properly, we may choose $B\in \PGL_{N+1}$ moving the coordinate axes to these hyperplanes, and consider $g=f^B$, which is ramified along the coordinate hyperplanes. Since the images of these hyperplanes under $g$ must also meet properly (by a similar induction on dimension), we may also choose $C\in \PGL_{N+1}$ so that $Cg$  fixes the coordinate hyperplanes, and then we must have $Cg (\mathbf{X})= D\mathbf{X}^d$ for some diagonal matrix $D$. But then $g(\mathbf{X})=A\mathbf{X}^d$ for $A=C^{-1}D$.
\end{proof}

The next lemma gives a description of the map $\PGL_{N+1}\to \mathsf{M}_d^N$ as a quotient by a group action. Sepcifically, let $G \subseteq \PGL_{N+1}$ be the subgroup generated by the diagonal and permutation matrices, acting on $\PGL_{N+1}$ by $D\cdot A= D^{-1}AD^d$ for diagonal matrices $D$, and by the usual conjugation for symmetric matrices. Then we have the following.

\begin{lemma}
The fibres of $\PGL_{N+1}\to \mathsf{M}_d^N$ are precisely the $G$-orbits.
\end{lemma}

\begin{proof}
	It is easy to check that change-of-variables on $\PP^N$ by elements of $G$ corresponds to this action of $G$ on $\PGL_{N+1}$. That is, if $M\in G$, then \[f_A^M=M^{-1}f_A  M=f_{M\cdot A},\] where $\cdot$ here denotes the action described above and we are identifying $M$ with the associated linear endomorphism of $\PP^N$. So it suffices to show that $G$ acts transitively on fibres. Suppose that $f_A$ and $f_B$ are conjugate, so that there exists an $M\in \PGL_{N+1}$ with $M f_B=f_A M$. Then $M$ must map the critical locus of $f_B$ to that of $f_A$, and in particular must therefore permute the coordinate hyperplanes. In other words, $M=DS$ for some permutation matrix $S$, and some $D\in\PGL_{N+1}$ which fixes the coordinate hyperplanes, and hence is diagonal, and so $M\in G$.
\end{proof}

Ultimately, we would like to be able to choose a representative of the $\PGL_{N+1}$-conjugacy class of $f_A$ with a certain normal form. For instance, every product of a permutation matrix with a diagonal matrix represents the conjugacy class of the power map, but the most obvious choice for a matrix in $\GL_{N+1}$ to represent this conjugacy class is the identity matrix. In order to capture the appropriate notion of normalization, we say that the matrix $A\in\GL_{N+1}$ is a \emph{normalized lift} of the $G$-orbit of its image in $\PGL_{N+1}$ (equivalently, of the conjugacy class of $f_A$) if and only if  every row of $A^{-1}$ contains a 1.
\begin{lemma}\label{lem:ones}
Every $G$-orbit in $\PGL_{N+1}(k)$ admits a normalized lift.
\end{lemma}

\begin{proof}
We claim that, for any $B\in \GL_{N+1}(k)$, there is a diagonal matrix $D\in\GL_{N+1}(k)$ such that every row of $D^{-d}BD$ contains a 1. Then for $A\in\PGL_{N+1}$, we may choose a lift of $A$ to $\GL_{N+1}$, and apply this claim to the the inverse of that lift to obtain the desired result.

If the $i,j$th entry of $B$ is $B_{i,j}$, and the $i$th diagonal entry of $D$ is $D_i$, then the $i,j$th entry of $D^{-d}BD$ will be $D_i^{-d}B_{i,j}D_j$. Since $B\in \GL_{N+1}$, we may choose a function $\sigma:\{0, ..., N\}\to \{0, ..., N\}$ such that $B_{i, \sigma(i)}\neq 0$ for each $1\leq i\leq N+1$. We will show that we may choose $D_i\in k$ so that \begin{equation}\label{eq:DiBij}D_{i}^{-d}B_{i, \sigma(i)}D_{\sigma(i)}=1\end{equation} for all $i$.

First, note that since $\{0, ..., N\}$ is a finite set, every element must be preperiodic under $\sigma$. Suppose for now that $j$ is periodic, say with $\sigma^m(j)=j$ and $m\geq 1$ minimal, and choose $D_j\in k$ so that
\begin{equation}\label{eq:Djchoose}D_j^{d^m-1}=\prod_{k=0}^{m-1}B_{\sigma^k(j), \sigma^{k+1}(j)}^{d^{m-1-k}}.\end{equation}
Once $D_j$ is so chosen, we will take
\[D_{\sigma^{k+1}(j)}=D_{\sigma^k(j)}^dB_{\sigma^k(j), \sigma^{k+1}(j)}^{-1}\]
for $0\leq k\leq m-2$ to ensure that~\eqref{eq:DiBij} is satisfied for $i\in \{\sigma(j), \sigma^2(j), ..., \sigma^{m-1}(j)\}$. We would like to know that~\eqref{eq:DiBij} is satisfied for $i=j=\sigma^m(j)$, as well.

Consider the quantity \[t_r=D_{\sigma^m(j)}D_{\sigma^r(j)}^{-d^{m-r}}\prod_{k=r}^{m-1}B_{\sigma^k(j), \sigma^{k+1}(j)}^{d^{m-1-k}}\\
\]
for $0\leq r\leq m-1$.
On the one hand, for $0\leq r\leq m-2$,
\begin{eqnarray*}
t_r & = &D_{\sigma^m(j)}D_{\sigma^r(j)}^{-d^{m-r}}\prod_{k=r}^{m-1}B_{\sigma^k(j), \sigma^{k+1}(j)}^{d^{m-1-k}}\\
&=& D_{\sigma^m(j)}(D_{\sigma^r(j)}^{-d}B_{\sigma^r(j), \sigma^{r+1}(j)})^{d^{m-r-1}}\prod_{k=r+1}^{m-1}B_{\sigma^k(j), \sigma^{k+1}(j)}^{d^{m-1-k}}\\
&=& D_{\sigma^m(j)}D_{\sigma^{r+1}(j)}^{-d^{m-(r+1)}}\prod_{k=r+1}^{m-1}B_{\sigma^k(j), \sigma^{k+1}(j)}^{d^{m-1-k}}\\
&=&t_{r+1}.
\end{eqnarray*}
On the other hand, since $\sigma^m(j)=j$, the construction~\eqref{eq:Djchoose} implies that $t_0=1$, and so it follows that
\[1=t_0=t_1=\cdots =t_{m-1}=D_{\sigma^m(j)}D_{\sigma^{m-1}(j)}^{-d}B_{\sigma^{m-1}(j), \sigma^m(j)},\]
establishing~\eqref{eq:DiBij} for $i=\sigma^m(j)=j$ as well.

To recap, we have shown that we may choose $D_i$ satisfying~\eqref{eq:DiBij} simultaneously for all $i$ in any given periodic cycle of $\sigma$, and since these cycles are disjoint, we may independently choose such $D_i$ for all periodic $i\in \{0, ..., N\}$. But now, if $i$ is not periodic, and $\sigma(i)$ has already been chosen, we may simply choose $D_i$ so that $D_i^d=B_{i, \sigma(i)}D_{\sigma(i)}$. Since every $i$ falls into a unique periodic cycle after a finite number of iterations of $\sigma$, we have eventually chosen $D_i$ for all $0\leq i\leq N$.
\end{proof}

Note, in the proof above, that if $B_{i, \sigma(i)}=1$ for all $i$ already, then the $D_i$ must all be $(d^{m}-1)$th roots of unity, and $m\leq N+1$. In particular, since there are at most $(N+1)^{N+1}$ ways to mark a position in each row of an element of $\GL_{N+1}$, there are at most a finite number of normalized representatives in each $G$-orbit, bounded just in terms of $d$ and $N$.


\section{Homogeneous Greens functions}\label{sec:local}

Let $K$ be an algebraically closed field, complete with respect to some absolute value $|\cdot|$. We will write
\[\|x_1, ..., x_k\|=\max\{|x_1|, ..., |x_k|\}.\]
Similarly, if $\Phi$ is a polynomial in any number of variables, then $\|\Phi\|$ will denote the application of this norm to the tuple of coefficients.

We will lift our map $f:\PP^N\to\PP^N$ to a homogeneous map $F:\AA^{N+1}\to\AA^{N+1}$. For maps of the form~\eqref{eq:Aform}, this amounts to choosing a representative $A\in \GL_{N+1}$ for a given matrix in $\PGL_{N+1}$, so for the remainder of this section our matrix $A$ has non-homogeneous coordinates.

Let $F_A$ be the map $\AA^{N+1}\to\AA^{N+1}$ by $\mathbf{X}\mapsto A\mathbf{X}^d$. For a homogeneous form $\Phi$ on $\AA^{N+1}$, we define
\[F_A^*\Phi(\mathbf{X})=\Phi\circ F_A(\mathbf{X})\]
and
\[{F_A}_*\Phi(\mathbf{Y})=\prod_{\mathbf{Y}=F_A(\mathbf{X})}\Phi(\mathbf{X}),\]
noting that ${F_A}_*{F_A}^*\Phi=\Phi^{d^{N+1}}$, $\deg(F_A^*\Phi)=d\deg(\Phi)$, and $\deg({F_A}_*\Phi)=d^N\deg(\Phi)$ (in this last case, the product on the right has degree $d^{N+1}$ in $\mathbf{X}$, and hence $d^N$ in $\mathbf{Y}$). If $\phi$ is the $d$th power map, and $A$ is the linear map associated to the matrix $A$, then note that $F_A^*=\phi^*A^*$, that ${F_A}_*=A_*\phi_*$, and that $A_*=(A^{-1})^*$.

For some intuition, and to guard against a possible misconception, we note that if $\Phi=0$ defines the divisor $D$ on $\PP^N$, then $F_A^*\Phi$ defines the divisor $f_A^*D$, but ${F_A}_*\Phi$ defines the divisor $\deg(f_A){f_A}_*D$, where in both cases $f_A$ is the endomorphism of $\PP^N$ defined by $F_A$.

Now, set
\begin{equation}\label{eq:greens}G_{F_A}(\Phi)=\lim_{k\to\infty}\frac{1}{d^{(N+1)k}}\log\|{F_A}_*^k\Phi\|.\end{equation}
Note that variants of this \emph{Greens function} appear throughout the literature (see, e.g., \cite{hp}), and that this definition agrees with that in~\cite{pnheights}. The basic properties of this function are as follows.  
\begin{lemma}\label{lem:greens}
The limit in~\eqref{eq:greens} exists, and satisfies
\begin{enumerate}
\item \label{it:func} $G_{F_A}({F_A}_*\Phi)=d^{N+1}G_{F_A}(\Phi)$,
\item \label{it:func2} $G_{F_A}({F_A}^*\Phi)= G_{F_A}(\Phi)$,
\item \label{it:scaling1} $G_{F_A}(\alpha\Phi)=G_{F_A}(\Phi)+\log|\alpha|$,
\item \label{it:scaling2} $G_{F_{\alpha A}}(\Phi)=G_{F_A}(\Phi)-\frac{\deg(\Phi)}{d-1}\log|\alpha|$,
\item\label{it:gupp} \begin{multline*}
G_{F_A}(\Phi)\leq \log\|\Phi\|+\frac{\deg(\Phi)}{d-1}\left(N\log\|A\|-\log|\det(A)|+\log^+|2N!|+Nd^{-N}\log^+|2|\right)\\+\frac{N}{d^{N+1}-1}\log^+|2|,  
 \end{multline*}
\item \label{it:glow}\[G_{F_A}(\Phi)\geq \log\|\Phi\|-\frac{\deg(\Phi)}{d-1}(\log\|A\|+\log^+|2|+Nd^{-N}\log^+|2|)-\frac{N}{d^{N+1}-1}\log^+|2|,\] 
\item \label{it:linear}$G_{F_A}(\Phi\Psi)=G_{F_A}(\Phi)+G_{F_A}(\Psi)$, and
\item \label{it:covariance} if $B\in G$ is a permutation-diagonal matrix, then $G_{F_A^B}(\Phi)=G_{F_A}(B_*\Phi)$.
\end{enumerate}
\end{lemma} 

\begin{proof}
Let $\phi$ be the $d$th power map, so that $F_A(\mathbf{X})=A\phi(\mathbf{X})$, where we are tacitly identifying the matrix $A$ with the associated linear map. First, note that by the triangle inequality
\begin{equation}\label{eq:linearpull}\log\|A^*\Phi\|=\log\|\Phi(A\mathbf{X})\|\leq \log\|\Phi\|+\deg(\Phi)\log\|A\|+\log^+\left|\binom{N+\deg(\Phi)}{N}\right|,\end{equation}
while
	\[\log^+\left|\binom{N+\deg(\Phi)}{N}\right|\leq \deg(\Phi)\log^+|2|+N\log^+|2|.\]
	Since $A^{-1}$ can be written as $\det(A)^{-1}A^{\mathrm{adj}}$, where every entry of $A^{\mathrm{adj}}$ is the determinant of some $N\times N$ submatrix of $A$, we have
	\begin{equation}\label{eq:inverselambda}\log\|A^{-1}\|\leq N\log\|A\|-\log|\det(A)|+\log^+|N!|\end{equation}
	and it follows that
	\begin{eqnarray*}
\log\|A_*\Phi\|&=&\log\|(A^{-1})^*\Phi\|\\&\leq& \log\|\Phi\|+\deg(\Phi)\log\|A^{-1}\|+\deg(\Phi)\log^+|2|+N\log^+|2|\\
		&\leq & \log\|\Phi\|+\deg(\Phi)N\log\|A\|-\deg(\Phi)\log|\det(A)|\\&&+\deg(\Phi)\log^+|2N!|+N\log^+|2|.
	\end{eqnarray*}
	But now, since $A_*A^*\Phi=\Phi=A^*A_*\Phi$, we may estimate
	\[\log\|\Phi\|\leq \log\|A_*\Phi\|+\deg(\Phi)\log\|A\|+\deg(\Phi)\log^+|2|+N\log^+|2| \]
	and
	\begin{multline*}
\log\|\Phi\|\leq \log\|A^*\Phi\|+ \deg(\Phi)N\log\|A\|-\deg(\Phi)\log|\det(A)|\\+\deg(\Phi)\log^+|2N!|+N\log^+|2|.		
	\end{multline*}
	
	Now, note that $\|\phi^*\Phi\|=\|\Phi\|$, since the homogeneous forms $\Phi$ and $\Phi\circ \phi$ have the same coefficients (albeit associated to different monomials) while the lemmas of Gau\ss~\cite[Lemma~1.6.3, p.~22]{bg} and Gelfond~\cite[Lemma~1.6.11, p.~27]{bg} give
	\begin{align*}
\log\|\phi_*\Phi\|&=\log\left\|\prod_{\zeta_0^d=\cdots=\zeta_N^d=1}\Phi(\zeta_0X_0, ..., \zeta_NX_N)\right\|\\
&=\log\prod_{\zeta_0^d=\cdots=\zeta_N^d=1}\|\Phi(\zeta_0X_0, ..., \zeta_NX_N)\|+\epsilon\\
&=d^{N+1}\log\|\Phi\|^{d^{N+1}}+\epsilon,	\end{align*}
for some $\epsilon$ with $|\epsilon|\leq N\deg(\Phi)\log^+|2|$,
	since multiplying a coefficient by a root of unity does not change its absolute value.
	 It follows that
	\begin{multline*}
	 -(\deg(\Phi)d^{-1}\log\|A\|+\deg(\Phi)d^{-1}\log^+|2|+d^{-(N+1)}N(1+\deg(\Phi))\log^+|2|)\\
		\leq d^{-(N+1)}\log\|{F_A}_*\Phi\|-\log\|\Phi\|
		\\ \leq \deg(\Phi)d^{-1}N\log\|A\|-\deg(\Phi)d^{-1}\log|\det(A)|+\deg(\Phi)d^{-1}\log^+|2N!|\\+d^{-(N+1)}N\deg(\Phi)\log^+|2|+d^{-1}N\log^+|2|
	\end{multline*}
The existence of the limit $G_{F_A}(\Phi)$, along with properties~\eqref{it:gupp} and~\eqref{it:glow}, follow from this and a standard telescoping-sum argument used by Tate and others, which we briefly recall. In general, if
\[\left|d^{-(N+1)}\log\|{F_A}_*\Phi\|-\log\|\Phi\|\right|\leq c_1\deg(\Phi)+c_2,\]
with $c_1$ and $c_2$ independent of $\Phi$, then
\begin{eqnarray*}
	\left|d^{-k(N+1)}\log\|{F_A}_*^k\Phi\|-\log\|\Phi\|\right|&\leq &\sum_{i=0}^{k-1}\left|\frac{\log\|{F_A}^{i+1}_*\Phi\|}{d^{(i+1)(N+1)}}-\frac{\log\|{F_A}^i_*\Phi\|}{d^{i(N+1)}}\right|\\
	&\leq &\sum_{i=0}^{k-1}d^{-i(N+1)}\left(c_1\deg({F_A}^i_*\Phi)+c_2\right)\\
	&\leq &\sum_{i=0}^{k-1}d^{-i(N+1)}\left(c_1d^{iN}\deg(\Phi)+c_2\right)\\
	&=&\frac{1-d^{-k}}{1-d^{-1}}c_1\deg(\Phi)+\frac{1-d^{-k(N+1)}}{1-d^{-(N+1)}}c_2.
\end{eqnarray*}
Now replacing $\Phi$ by ${F_A}_*^m\Phi$ and dividing both sides by $d^{m(N+1)}$, we see that for any fixed $\Phi$, the sequence on the right of~\eqref{eq:greens} is Cauchy, and hence converges. Properties~\eqref{it:gupp} and~\eqref{it:glow} follow from the same calculation, taking $k\to \infty$ and using the values of $c_1$ and $c_2$ from the appropriate one-sided bounds. Property~\eqref{it:func} now follows directly from the limit definition. 

For property~\eqref{it:scaling1} we need only note that ${F_A}_*(\alpha \Phi)=\alpha^{d^{N+1}}\Phi$ (directly from the definition), while item~\eqref{it:scaling2} 
follows from this along with the fact that \[{F_{\alpha A}}_*\Phi(\mathbf{X})=F_A\Phi(\alpha^{-1}\mathbf{X})={F_A}_*\alpha^{-\deg(\Phi)}\Phi(\mathbf{X})=\alpha^{-\deg(\Phi)d^{N+1}}{F_A}_*\Phi,\] again from the definition. 

Item~\eqref{it:linear} follows from Gau\ss' Lemma when the absolute value is non-archimedean, and from a comparison to Mahler's measure when $|\cdot|$ is the usual absolute value on $\CC$. Specifically, if 
\[m(\Phi)=\int_{\AA_\CC^{N+1}}\log|\Phi(z_0, ..., z_N)|d\mu(z_0)\cdots d\mu(z_N)\]
where $\mu$ is Lebesgue measure on the unit circle normalized to give circumference 1, then estimates of Mahler~\cite{mahler} give
\[\log\|\Phi\|=m(\Phi)+O(\deg(\Phi)),\]
where the implied constant depends only on $N$. Concretely,
\[m(\Phi)-\frac{N}{2}\log(\deg(\Phi)+1)\leq \log\|\Phi\|\leq m(\Phi)+N\deg(\Phi)\log 2.\]
 Then, since ${F_A}_*$ is multiplicative on homogeneous forms,
\begin{eqnarray*}
d^{-k(N+1)}\log\|{F_A}^k_*(\Phi\Psi)\| &=&	d^{-k(N+1)}m({F_A}^k_*(\Phi\Psi))\\&&+O\left(d^{-k(N+1)}\deg({F_A}^k_*(\Phi\Psi)\right)\\
&=&d^{-k(N+1)}m({F_A}^k_*\Phi)+d^{-k(N+1)}m({F_A}^k_*\Psi)\\&&+O(\left(d^{-k}\deg(\Phi\Psi)\right)\\
&=&d^{-k(N+1)}\log\|{F_A}^k_*\Phi\|+d^{-k(N+1)}\log\|{F_A}^k_*\Psi\|\\&&+O\left(d^{-k(N+1)}(\deg({F_A}_*^k\Phi)+\deg({F_A}^k_*\Psi))\right)\\&&+O(\left(d^{-k}\deg(\Phi\Psi)\right)\\
&=&d^{-k(N+1)}\log\|{F_A}^k_*\Phi\|+d^{-k(N+1)}\log\|{F_A}^k_*\Psi\|+o(1),
\end{eqnarray*}
where $o(1)\to 0$ as $k\to\infty$, which gives~\eqref{it:linear} in the archimedean case (in the non-archimedean case, $\|\cdot\|$ is already multiplicative).
 
Item~\eqref{it:func2} follows from~\eqref{it:func} and~\eqref{it:linear}:
\[G_{F_A}({F_A}^*\Phi)=\frac{1}{d^{N+1}}G_{F_A}({F_A}_*{F_A}^*\Phi)=\frac{1}{d^{N+1}}G_{F_A}(d^{N+1}\Phi)=G_{F_A}(\Phi).\]

Finally, for item~\eqref{it:covariance}, note that $(F_A^B)^k_*\Phi = B^{-1}_*{F_A}^k_*B_*\Phi$, and so the estimates above give
\[\log\|(F_A^B)^k_*\Phi\|=\log\|{F_A}^k_*B_*\Phi\|+O(1),\]
where the implied constant depends on $B$. Dividing by $d^{k(N+1)}$ and letting $k\to \infty$ eliminates the error term.
\end{proof}


\section{Lyapunov exponents and moduli space}\label{sec:lyap}

We retain the notation of the previous section, and define
\begin{equation}\label{eq:lyapdef}L(f_A)=G_{F_A}(J_{F_A})-\log|d|,\end{equation}
where $J_{F_A}$ is the usual Jacobian determinant of $F_A$, computed here as
\[J_{F_A}=d^{N+1}\det(A)\prod_{i=0}^NX_i^{d-1}.\]
For any scalar $\alpha$, Lemma~\ref{lem:greens} gives
\begin{eqnarray*}
G_{F_{\alpha A}}(J_{F_{\alpha A}})&=&G_{F_{A}}(\alpha^{N+1}J_{F_{A}})-\frac{\deg(J_{F_{\alpha A}})}{d-1}\log|\alpha|\\
&=&	G_{F_A}(J_{F_A})+(N+1)\log|\alpha|-\frac{(N+1)(d-1)}{d-1}\log|\alpha|\\
&=&G_{F_A}(J_{F_A}),
\end{eqnarray*}
 and so the definition~\eqref{eq:lyapdef} is independent of the choice of lift of $A\in \PGL_{N+1}$ to $\GL_{N+1}$.
\begin{lemma}\label{lem:lyaprel} 
In the case $K=\CC$,
 $L(f_A)$ is the sum of Lyapunov exponents of $f_A$ relative to its measure of maximal entropy.
\end{lemma}

\begin{proof}
	This is present in~\cite{pnheights}, but we sketch a quick proof here for completeness. Choose a lift $F$ of $f_A$, let $\mu_F$ be the measure of maximal entropy associated to $F$, for a homogeneous form $\Phi$ let
	\[\nu_F(\Phi)=\int_{\AA_\CC^{N+1}}\log|\Phi(z)|\mu_F.\]	
	 First, note from~\cite[Lemma~1.4]{berteloot} it suffices to show that $\nu_F(J_H)=G_F(J_F)$, where $J_F$ is the usual Jacobian determinant of $F$. More generally, we will show that $\nu_F(\Phi)=G_F(\Phi)$ for any homogeneous form $\Phi$. First, note that
	 \begin{align*}
	 	G_F(\Phi)&=\log\|\Phi\|+O(\deg(\Phi))\\
	 	&=\int_{\CC^{N+1}}\log|\Phi|(dd^c\log^+\|\cdot\|)^{N+1}+O(\deg(\Phi))\\
	 	&=\nu_F(\Phi)+O(\deg(\Phi))
	 \end{align*}
	 by the aforementioned estimates of Mahler~\cite{mahler} and the fact that $G_F=\log\|\cdot\|+O(1)$. On the other hand, the function $z\mapsto |\Phi(z)|$ satisfies $F_*|\Phi|=|F_*\Phi|$, and so
	 \[\nu_F(F_*\Phi)=\int_{\CC^{N+1}}F_*\log|\Phi|\mu_F=\int_{\CC^{N+1}}\log|\Phi|F^*\mu_F=d^{N+1}\nu_F(\Phi),\] since $F^*\mu_F=d^{N+1}\mu_F$.
	 From this we have
	 \begin{align*}
G_F(\Phi)&=d^{-k(N+1)}G_F(F_*^k\Phi)\\
&=d^{-k(N+1)}\nu_F(F_*^k\Phi)+d^{-k(N+1)}O(\deg(F_*^k\Phi))\\
&=\nu_F(\Phi)+O(d^{-k}\deg(\Phi)) 	
	 \end{align*}
	 for all $k$, and hence $G_F(\Phi)=\nu_F(\Phi)$ as claimed.
\end{proof}

%

Our next lemma involves an estimate on the local heights of the push-forwards of the coordinate hyperplanes by a linear map. For the purposes of this lemma, given $A\in\GL_{N+1}$, let $M_{i, j}$ denote the $i, j$th minor of $A$, that is, the determinant of the submatrix obtained by deleting from $A$ row $i$ and column $j$.
\begin{lemma}\label{lem:goodrep}
Every $G$-orbit in $\GL_{N+1}(K)$ other than that of the identity matrix contains a matrix $A$ with $M_{0, 1}(A)\neq 0$ and
\begin{multline}
\log\|A\|+N^2\log|M_{1, 0}(A)|\leq \\ N\sum_{i=0}^{N} \log\|A_*X_i\|+(1+N^2)\log|\det(A)|+\log^+|N!|.	
\end{multline}
\end{lemma}

\begin{proof}
First, note that while neither side of the inequality is well-defined on $\PGL_{N+1}$, both are homogeneous of the same degree. In particular, if $A\in \GL_{N+1}$ and $\alpha\in K^*$, then
\[\log\|\alpha A\|=\log\|A\|+\log|\alpha|,\]
while
\[\log\|(\alpha A)_*X_i\|=\log\|A_*X_i\|-\log|\alpha|,\]
\[\log|\det(\alpha A)|=\log|\det(A)|+(N+1)\log|\alpha|,\]
and
\[\log|M_{i, j}(\alpha A)|=\log|M_{i, j}(A)|+N\log|\alpha|.\]

By Lemma~\ref{lem:ones} we may choose a lift of $A$ to $\GL_{N+1}$ so that $B=A^{-1}$ has a 1 on every row. Furthermore, if $A$ is not diagonal, then neither is $B$, and so if $B_{i,j}$ is the $i, j$th entry of $B=A^{-1}$, we may suppose that there is some $I\neq J$ with $B_{I, J}=1$ (the proof of Lemma~\ref{lem:ones} involves only scaling arbitrary non-zero entries to be 1). Conjugating by a permutation matrix, we may in fact choose a lift with $B_{0, 1}=1$.

Now
\begin{eqnarray}
\log\|A^{-1}\|&=&\log\max_{i=0}^{N}\max_{j=0}^{N}|B_{i, j}|\nonumber\\
&\leq & \sum_{i=0}^{N}\log\max_{j=0}^{N}|B_{i, j}|\label{eq:tricky}\\
&=& \sum_{i=0}^{N}\log\|(A^{-1})^* X_i\|\nonumber\\
&=& \sum_{i=0}^{N}\log\|A_* X_i\|\nonumber
\end{eqnarray}
where the crucial step~\eqref{eq:tricky} follows from the fact that each summand is non-negative.  We also have
\[\log\|A\|=\log\|(A^{-1})^{-1}\|\leq N\log\|A^{-1}\|+\log|\det(A)|+\log^+|N!|,\]
and so all together we have
\begin{eqnarray*}
	\log\|A\|&\leq & N\sum_{i=0}^{N}\log\|A_* X_i\|+\log|\det(A)|+\log^+|N!|\\
	&=&N\sum_{i=0}^{N}\log\|A_* X_i\|+\log|\det(A)|+\log^+|N!|-N^2\log|B_{0, 1}|\\
	&=&N\sum_{i=0}^{N}\log\|A_* X_i\|+(1+N^2)\log|\det(A)|-N^2\log|M_{1, 0}(A)|\\&&+\log^+|N!|,
\end{eqnarray*}
using the facts that $B_{0, 1}=1$ and that $\det(A)B_{0, 1}=(-1)^{I+J}M_{1, 0}(A)$.
\end{proof}

The following lemma gives our main estimates on the Lyapunov exponents, used in the next section to prove the main results.

\begin{lemma}\label{lem:lyapbound}
For any  $A\in\GL_{N+1}$ we have
	\begin{multline*}
	L(f_A)\leq N(N+1)\log\|A\|-N\det(A)+(N+1)\log^+|2N!|\\+\frac{N(N+1)(2d^{N+1}-d^N-1)}{d^N (d^{N+1}-1)}\log^+|2|+N\log|d|,		
	\end{multline*} 
while for $A$ as constructed in Lemma~\ref{lem:goodrep}, we have
	\begin{multline}\label{eq:lyaplow}L(f_A)\geq \left(\frac{d-(N^2+N+1)}{dN}\right)\log\|A\|+\frac{N(d-1)}{d}\log|M_{1, 0}(A)|\\+\frac{dN-(N^2+1)(d-1)}{dN}\log|\det(A)|
-\frac{d-1}{dN}\log^+|N!|\\-\frac{N+1}{d}\left(\frac{d^N+N}{d^N}+\frac{N(d-1)}{d^{N+1}-1}\right)\log^+|2|+N\log|d|.\end{multline}
\end{lemma}

\begin{proof}
Note that $J_{F_A}=d^{N+1}\det(A)\prod_{i=0^N}X_i^{d-1}$, and so
\[G_{F_A}(J_{F_A})=(d-1)\sum_{i=0}^NG_{F_A}(X_i)+(N+1)\log|d|+\log|\det(A)|.\]
For the first inequality, observe simply that
\begin{eqnarray*}
L(f_A)&=&(d-1)\sum_{i=0}^NG_{F_A}(X_i)+\log|\det(A)|+N\log|d|\\
&\leq &(N+1)\left(N\log\|A\|-\log|\det(A)|+\log^+|2N!|+Nd^{-N}\log^+|2|\right)\\&&+\frac{N(N+1)(d-1)}{d^{N+1}-1}\log^+|2|+\log|\det(A)|+N\log|d|
\end{eqnarray*}
by Lemma~\ref{lem:greens}\eqref{it:gupp}, since $\log\|X_i\|=0$.

With $A$ as constructed in Lemma~\ref{lem:goodrep}, 
\begin{eqnarray*}
	\sum_{i=0}^N G_{F_A}(A_* X_i)&\geq & \sum_{i=0}^N \log\|A_* X_i\|-\frac{N+1}{d-1}(\log\|A\|+\log^+|2|+Nd^{-N}\log^+|2|)\\ &&\quad -\frac{N(N+1)}{d^{N+1}-1}\log^+|2|\\
	&\geq &\left(\frac{1}{N}-\frac{N+1}{d-1}\right)\log\|A\|+N\log|M_{1, 0}(A)|-\frac{1+N^2}{N}\log|\det(A)|\\&&-\frac{1}{N}\log^+|N!|-\left(\frac{(N+1)(1+Nd^{-N})}{d-1}+\frac{N(N+1)}{d^{N+1}-1}\right)\log^+|2|\\
	\end{eqnarray*}

Note also that since $\phi_* X_i=(-1)^{d+1}X_i^{d^N}$  we have  ${F_A}_*X_i=(-1)^{d+1}(A_* X_i)^{d^N}$, so

\begin{eqnarray*}
L(f_A)&=&G_{F_A}(J_{F_A})-\log|d|\\
&=&(d-1)\sum_{i=0}^NG_{F_A}(X_i)+\log|\det(A)|+N\log|d|\\
&=&\frac{d-1}{d^{N+1}}\sum_{i=0}^NG_{F_A}({F_A}_*X_i)+\log|\det(A)|+N\log|d|\\
&=&\frac{d-1}{d}\sum_{i=0}^NG_{F_A}(A_*X_i)+\log|\det(A)|+N\log|d|	\\
&\geq& \frac{d-1}{d}\left(\frac{1}{N}-\frac{N+1}{d-1}\right)\log\|A\|+\frac{N(d-1)}{d}\log|M_{1, 0}(A)|\\&&+\frac{dN-(N^2+1)(d-1)}{dN}\log|\det(A)|-\frac{d-1}{dN}\log^+|N!|\\
&&-\frac{N+1}{d}\left(\frac{d^N+N}{d^N}+\frac{N(d-1)}{d^{N+1}-1}\right)\log^+|2|+N\log|d|\\ 
\end{eqnarray*}
We now complete the proof of the first inequality by noting 
\[\frac{d-1}{d}\left(\frac{1}{N}-\frac{N+1}{d-1}\right)=\frac{d-(N^2+N+1)}{dN}.\]

\end{proof}

We note that the upper bound on $L(f_A)$ is very natural, and is well-defined on $\PGL_{N+1}$. Indeed, \[\lambda_{\PGL_{N+1}}(A)=-\frac{1}{N+1}\log|\det(A)|+\log\|A\|\] is the N\'{e}ron function on $\PP^{(N+1)^2-1}\supseteq\PGL_{N+1}$ relative to the divisor defined by the vanishing of the determinant (normalized to have degree 1) and the standard metric. Since $L(f_A)$ depends only on the conjugacy class of $f_A$, the second inequality gives
\[L(f_A)\leq N(N+1)\inf_{B\in G\cdot A}\lambda_{\PGL_{N+1}}(B)+O(1),\]
and the infimum above is a natural measure of the proximity of the class of $f_A$ to the boundary of $\MinCrit_d^N$.

The lower bound in Lemma~\ref{lem:lyapbound} also references the distance to the vanishing of some off-diagonal minor, and it is not clear that it has a natural interpretation in terms of pluripotential theory on $\PGL_{N+1}$.

For our next lemma, which will be used in the next section to relate the global heights on $\PGL_{N+1}$ and $\mathsf{M}_d^N$, it is useful to introduce a N\'{e}ron function on $\Hom_d^N$. For any $f\in \Hom_d^N$ choose a lift $F$ (that is, an affine map given by a particular choice of representatives for the homogeneous coefficients of $f$), and set
\[\lambda_{\Hom_d^N}(f)=-\frac{1}{(N+1)d^N}\log|\Res(F)|+\log\|F\|,\]
which is independent of the choice of lift. Note that this is just the usual N\'{e}ron function on the affine variety $\Hom_d^N$ with respect to the standard metric, normalized for the degree of the hypersurface at infinity.

\begin{lemma}\label{lem:coc}
Let $B\in\PGL_{N+1}(K)$. Then
\[\lambda_{\operatorname{Hom}_d^N}(f^B)\leq \lambda_{\operatorname{Hom}_d^N}(f)+ (d+N)\lambda_{\PGL_{N+1}}(B)+(d+N)\log^+|2|+\log^+|(N+1)!|.\]
\end{lemma} 

\begin{proof}
We suppose that we have chosen affine lifts of $f$ and $B$, called $F$ and $B$, so that we may speak of their coordinates.

First, note that  we have
\begin{multline*}
\log\|B^{-1}\circ F\|\leq \log\|B^{-1}\|+\log\|F\|+\log^+|N+1|\\\leq \log\|F\|+ N\log\|B\|-\log|\det(B)|+\log^+|(N+1)!|	
\end{multline*}
by the triangle inequality and~\eqref{eq:inverselambda}.

On the other hand, for each coordinate function $F_i$ we have
\[\log\|F_i\circ B\|\leq \log\|F_i\|+\deg(F_i)\log\|B\|+(\deg(F_i)+N)\log^+|2|\]
by~\eqref{eq:linearpull} above. So
\[\log\|F\circ B\|\leq \log\|F\|+d\log\|B\|+(d+N)\log^+|2|.\]

Finally, we have
\[\Res(F^B)=\Res(F)\det(B)^{d^N(d-1)},\]
by \cite[Theorem~3.13, p.~399]{lang} and  \cite[Corollary~3.14, p.~400]{lang}, so
\begin{eqnarray*}
	\lambda_{\Hom_d^N}(f^B)&=&-\frac{1}{(N+1)d^N}\log|\Res(F^B)|+\log\|F^{B}\|\\
	&\leq &-\frac{1}{(N+1)d^N}\log|\Res(F)|-\frac{d-1}{N+1}\log|\det(B)|\\
	&&+\log\|F\circ B\|+ N\log\|B\|-\log|\det(B)|+\log^+|(N+1)!|\\
	&\leq & -\frac{1}{(N+1)d^N}\log|\Res(F)|+\log\|F\|
	\\&&-\frac{N+1+d-1}{N+1}\log|\det(B)|+(d+N)\log\|B\|\\&&+(d+N)\log^+|2|+\log^+|(N+1)!|\\
	&=&\lambda_{\Hom_d^N}(f)+(d+N)\lambda_{\PGL_{N+1}}(B)+(d+N)\log^+|2|\\&&+\log^+|(N+1)!|.
\end{eqnarray*}
\end{proof}


\section{Global fields, proofs of the main results}\label{sec:global}

In this section, we let $K$ denote a \emph{product formula field}, that is, a field with a set of inequivalent valuations $M_K$ with associated absolute values $|\cdot|_v$ and weights $n_v$ such that for any $\alpha\neq 0$ from $K$, we have $|\alpha|_v=1$ for all but finitely many $v\in M_K$, and
\[\prod_{v\in M_K}|\alpha|_v^{n_v}=1.\]
The archetypal example for us will be when $K$ is a number field, $M_K$ is the usual set of absolute values (that is, the set of absolute values extending the usual and $p$-adic absolute values on $\QQ\subseteq K$), and $n_v=[K_v:\QQ_v]/[K:\QQ]$. Another key example comes from algebraic geometry, where any finitely generated extension $K$ of an algebraically closed field $k$ (a \emph{function field}) comes equipped with a natural set of absolute values $M_K$, which are nontrivial if $K\neq k$. Specifically, any such field is the function field of a normal projective variety $X/k$~\cite[Lemma~1.4.10, p.~12]{bg}, and for every prime divisor $D$ on $X$ we may construct an absolute value on $K$ by
\[|f|_D=e^{-\ord_D(f)}.\]
With weights $n_D=\deg(D)$ these satisfy the product formula. Note that $|f|_D\leq 1$ for all $D$ if and only if $f\in k(X)$ is constant.

For every $v\in M_K$ there is an algebraically closed field $\CC_v\supseteq K$ which is complete with respect to some extension of $|\cdot|_v$, and quantities from Section~\ref{sec:local} and~\ref{sec:lyap} constructed relative to this completion are here distinguished with a subscript $v$.
For $\alpha\in K$ we define the height by
\[h(\alpha)=\sum_{v\in M_K}n_v\log^+|\alpha|_v,\]
and similarly for heights in projective space.
Note that if $L/K$ is a finite separable extension of fields, if $M_L$ is the set of absolute values on $L$ extending elements of $M_K$, and if $n_w=n_v[L_w:K_v]/[L:K]$ for an absolute value $w\in M_L$ extending $v\in M_K$, then the height as defined above is independent of the field in which it is computed~\cite[Chapter~1]{bg}. We will, then, allow ourselves to make finite separable extensions of the base field at any time (for instance in our invocation of Lemma~\ref{lem:goodrep} in the following proof), effectively working over the separable closure of $K$.

\begin{proof}[Proof of Theorems~\ref{th:main} and~\ref{th:exp}]
For any polarized endomorphism $f$ of $\PP^N$, let $h_{\mathcal{O}(1)}$ be the canonical height function for pure cycles on $\PP^N$ introduced by Zhang~\cite{zhang}, and set
\[\hat{h}_f(D)=\deg(D)h_{\mathcal{O}(1)}(D)\]
for divisors $D$ on $\PP^N$. Note that the properties of Zhang's height~\cite[Theorem~2.4]{zhang} now give
\begin{gather*}
\hat{h}_f(f_*D)=\deg(f)^N\hat{h}_f(D),\\ 
\hat{h}_f(D)=h(D)+O_f(\deg(D)),
\intertext{ where $h(D)=\sum_{v\in M_K}n_v\log\|\Phi\|_v$ for any homogeneous form $\Phi$ on $\AA^{N+1}$ whose vanishing defines $D$ (see~\cite{pnheights} for a uniform estimate on the error),
as well as }
\hat{h}_f(nD)=n\hat{h}_f(D)
\end{gather*}
(indeed, the third property follows from the first two). We now define
\[\hat{h}_{\mathrm{crit}}(f)=\hat{h}_f(C_f),\]
where $C_f$ is the critical divisor of $f$.
This is not the same ``critical height'' as defined in~\cite{pcfpn}, although there is a relation (see~\cite{pnheights} for more on this).

For $A\in \PGL_{N+1}(K)$, choose a lift to $\GL_{N+1}$ (also denoted $A$), and let $F_A(\mathbf{X})=A\mathbf{X}^d$ be the affine endomorphism of $\AA^{N+1}$ corresponding to $f_A$, recalling that we are free to make a finite separable extension of the base field when doing this. For any homogeneous form $\Phi$ whose vanishing defines a divisor $D$ on $\PP^N$,
we have from Lemma~\ref{lem:greens} that
\begin{eqnarray*}
\sum_{v\in M_K}n_vG_{F_A, v}(\Phi)&=&\sum_{v\in M_K}n_v\left(\log\|\Phi\|_v+O_v(\deg(D))\right)\\
&=&h(D)+O(\deg(D))\\&=&\hat{h}_{f_A}(D)+O(\deg(D)),	
\end{eqnarray*}
where the implied constant depends on $F_A$ (note in particular that the error terms vanish for all but finitely many $v\in M_K$, justifying our summing them). But now ${F_A}_*^k\Phi$ defines $d^k{f_A}_*^kD$, and so applying the above to this homogeneous form gives
\begin{eqnarray*}
\sum_{v\in M_K}n_vG_{F_A, v}(\Phi)&=&d^{-k(N+1)}\sum_{v\in M_K}n_vG_{F_A}({F_A}_*^k\Phi)\\
&=&d^{-k(N+1)}\hat{h}_{f_A}(d^k{f_A}_*^kD)\\&&+d^{-k(N+1)}O(\deg(d^k{f_A}_*^kD))\\
&=&	\hat{h}_{f_A}(D)+d^{-k}O(\deg(D)),
\end{eqnarray*}
since $\deg({f_A}_*D)=d^{N-1}\deg(D)$. Taking $k\to\infty$, we obtain
\[\hat{h}_{f_A}(D)=\sum_{v\in M_K}n_vG_{F_A, v}(\Phi).\]
Note that we could equivalently take this as the definition of the canonical height, establish the properties above, and deduce \emph{post hoc} the relation to Zhang's height.

Writing $J_{F_A}=\det(DF_A)$ for the usual Jacobian determinant, whose vanishing defines $C_{f_A}$,
\begin{eqnarray*}
	\sum_{v\in M_K}n_vL_v(f_A)&=&\sum_{v\in M_K}n_vG_{F_A, v}(J_{F_A})+\sum_{v\in M_K}n_v\log|d|_v\\
	&=& \hat{h}_{f_A}(C_{f_A})\\
	&=&\hat{h}_{\mathrm{crit}}(f_A),
\end{eqnarray*}
by the product formula.

It now follows from Lemma~\ref{lem:lyapbound} that
\begin{eqnarray*}
	\hat{h}_{\mathrm{crit}}(f_A)&=&\sum_{v\in M_K}n_vL_v(f_A)\\
	&\leq &N(N+1)\sum_{v\in M_K}n_v\log\|A\|_v\\&&-N\sum_{v\in M_K}n_v\log|\det(A)|_v\\&& +(N+1)\sum_{v\in M_K}n_v\log^+|2N!|_v\\&&+\frac{N(N+1)(d-1)}{d^{N+1}-1}\sum_{v\in M_K}n_v\log^+|2|_v\\&&+N\sum_{v\in M_K}n_v\log|d|_v\\
	&=&N(N+1)h_{\PGL_{N+1}}(A)+(N+1)h (2N!)\\&&+\frac{N(N+1)(d-1)}{d^{N+1}-1}h(2).
\end{eqnarray*}
Of course, if $K$ is a number field, then $h(2)=\log 2$, and similarly for $h(2N!)$.

Before completing this direction of the proof, we introduce one more lemma, whose proof we defer.
\begin{lemma}\label{lem:conjugate}
	Let $f:\PP^N_K\to \PP^N_K$ be minimally critical of degree $d\geq 2$. Then $f\sim f_A$ for some $A$ with
\begin{multline*}
h_{\PGL_{N+1}}(A)\leq
(1+(d+N)(N+1))h_{\Hom_d^N}(f)+(d+N)h(N!)+(d+N)(N+1)h(d)\\+(d+N+1)h((N+1)!)+(d+N)(1+2N(N+1)(d-1))h(2)	
\end{multline*}
\end{lemma}

Assuming this lemma (which we prove below), and by the invariance of $\hat{h}_{\mathrm{crit}}(f_A)$ under conjugation, we may apply~\cite[Lemma~6.32, p.~102]{barbados} to obtain
\[\hat{h}_{\mathrm{crit}}(f)\leq N(N+1)\inf_{f_A\sim f}h_{\PGL_{N+1}}(A)+O_{K, N, d}(1)\ll h_{\mathsf{M}_d^N}(f),\]
for any minimally critical $f:\PP^N_K\to\PP^N_K$.
Indeed, this direction is known for endomorphims of $\PP^N$ in general~\cite{pnheights}.

In the other direction, assume that $A\in\GL_{N+1}$ is a normalized lift of its $G$-orbit in $\PGL_{N+1}$ (recalling that by Lemma~\ref{lem:goodrep}, every orbit has such a lift). Then in Lemma~\ref{lem:lyapbound} we obtain~\eqref{eq:lyaplow} in each absolute value. Taking the sum of both sides over all absolute values, weighted by the $n_v$, and noting that any term of the form $\log|\alpha|_v$ will vanish in this sum (by the product formula) we have
\begin{eqnarray*}
	\hat{h}_{\mathrm{crit}}(f_A)&=&\sum_{v\in M_K}n_vL_v(f_A)\\
	&\geq &\left(\frac{d-(N^2+N+1)}{dN}\right)h_{\PGL_{N+1}}(A)-\frac{d-1}{dN}h(N!)\\&&-\frac{N+1}{d}\left(\frac{N+d^N}{d^N}+\frac{N(d-1)}{d^{N+1}-1}\right)h(2).
\end{eqnarray*}
Note also that since $B\mapsto f_B$ induces a morphism $\PGL_{N+1}\to \mathsf{M}_d^N$, we have $h_{\mathsf{M}_d^N}(f_B)\ll h_{\PGL_{N+1}}(B)$ by the triangle inequality (where the implied constants depend on $d$ and $N$; one could also again apply~\cite[Lemma~6.32, p.~102]{barbados}).
Since every $G$-orbit admits a normalized lift, it follows that
\[\hat{h}_{\mathrm{crit}}(f_A)\geq \left(\frac{d-(N^2+N+1)}{dN}\right)\inf_{f_B\sim f_A}h_{\PGL_{N+1}}(B)-O_{K, d, N}(1)\gg h_{\mathsf{M}_d^N}(f_A),\] by Lemma~\ref{lem:conjugate}
as long as $d>N^2+N+1$, since both sides are now well-defined on moduli space.
\end{proof}

\begin{proof}[Proof of Lemma~\ref{lem:conjugate}]
	Summing the inequality in Lemma~\ref{lem:coc} over all places, we see that for any endomorphism $f:\PP^N\to\PP^N$ and $B\in \PGL_{N+1}$, we have
\[h_{\Hom_d^N}(f^B)\leq h_{\operatorname{Hom}_d^N}(f)+ (d+N)h_{\PGL_{N+1}}(B)+(d+N)h(2)+h((N+1)!)\]
	We also know that if $f$ is minimally critical, there is a $B$ such that $f^B$ has the form~\eqref{eq:Aform}, so all that remains is to bound the height of that matrix $B$. 
		
	Let $F$ be a lift of $f$. Since $f$ is minimally critical, for $J_F$ the usual Jacobian determinant, we have
\[J_F(\mathbf{X})=\prod_{i=0}^NL_i(\mathbf{X})^{d-1}.\]	
for some linear forms $L_i$, which are of course not uniquely determined by this relation, but whose coefficients we fix. Note that by above-cited estimates of Mahler~\cite{mahler} which give approximate multiplicativity of $\|\cdot\|$ in the archimedean case (and using Gau\ss' Lemma in the non-archimedean setting), we have for any $v\in M_K$
\begin{eqnarray*}
\sum\log\|L_i\|_v&\leq& \log\left\|\prod_{i=0}^NL_i^{d-1}\right\|+2N(N+1)(d-1)\log^+|2|_v\\
&=&\log\left\|\det\left(\frac{\partial F_i}{\partial X_j}\right)\right\|_v+2N(N+1)(d-1)\log^+|2|_v\\
&\leq& (N+1)\log\max\left\|\frac{\partial F_i}{\partial X_j}\right\|_v+\log^+|(N+1)!|_v\\&&+2N(N+1)(d-1)\log^+|2|_v\\
&\leq & (N+1)\log\|F\|_v+(N+1)\log^+|d|_v+\log^+|(N+1)!|_v\\&&+2N(N+1)(d-1)\log^+|2|_v
\end{eqnarray*}
since the entries of $J_F$ are partial derivatives of the forms defining $F$.
Summing over all places, we have
\begin{multline*}
\sum_{i=0}^{N+1}h(L_i)\leq (N+1)h_{\Hom_d^N}(f)\\+(N+1)h(d)+h((N+1)!)+2N(N+1)(d-1)h(2).	
\end{multline*}

 Now, let $L$ be the matrix whose $i$th row consists of the coefficients of $L_i$, let $D$ be any diagonal matrix, and let $B=L^{-1}D^{-1}$. Since $J_{F^B}=B^*J_F$, we have
\[J_{F^B}=\alpha\prod_{i=0}^{N+1} X_i^{d-1},\]
for some scalar $\alpha$, 
whence $F^B=A\mathbf{X}^d$ for some $A\in \GL_{N+1}$. So any matrix $B$ of this form will suffice.

Note that if $M\in \GL_{N+1}$ is a matrix with rows corresponding to linear forms $M_0, ..., M_N$, it is not in general true that \begin{equation}\label{eq:rows}h_{\PGL_{N+1}}(M)\leq \sum_{i=0}^{N+1} h(M_i),\end{equation}
for instance because the right-hand-side will vanish if $M$ is diagonal.
But~\eqref{eq:rows} becomes true if we insist that each row contain a 1. To see this, just note that in each absolute value $|\cdot|_v$, we would then have
\[\log\|M\|_v=\max_{0\leq i \leq N}\log\|M_i\|_v\leq \sum_{i=0}^{N+1}\log\|M_i\|_v.\]
In particular, we can choose $D$ so that each row of $DL$ contains a 1, and obtain
\begin{eqnarray*}
h_{\PGL_{N+1}}(B)&\leq & Nh_{\PGL_{N+1}}(B^{-1})+h(N!)\\
&= & Nh_{\PGL_{N+1}}(DL)+h(N!)\\
&\leq& \sum_{i=0}^N h(D_iL_i) + h(N!)\\
&\leq &(N+1) h_{\Hom_d^N}(f)+O_{d, N}(1),
\end{eqnarray*}
since $h(D_iL_i)=h(L_i)$.
We we have proven the result with \begin{multline*}
 	h_{\PGL_{N+1}}(B)\leq (N+1)h_{\Hom_d^N}(f)\\+h(N!)+(N+1)h(d)+h((N+1)!)+2N(N+1)(d-1)h(2),\end{multline*} and hence
\[h_{\PGL_{N+1}}(A)\leq (1+(N+d)(N+1))h_{\Hom_d^N}(f)+O_{d, N}(1).\]
Note that the $O_{d, N}(1)$ term involves only heights of integers, and so vanishes in the function field setting.
\end{proof}

\begin{proof}[Proof of Corollary~\ref{th:heightbound}]
Let $K$ be a number field. From the proof of Theorem~\ref{th:exp}, if $f:\PP^N\to\PP^N$ is minimally critical of degree $d> N^2+N+1$, then $f$ is conjugate to $f_A$ for some $A\in \PGL_{N+1}(\overline{K})$ satisfying
\begin{multline*}
h_{\PGL_{N+1}}(A)\leq \frac{d-1}{d-(N^2+N+1)}\log N!\\+\left(\frac{N(N+1)((d^{N+1}-1)+N(d-1))}{(d^{N+1}-1)(d-(N^2+N+1))}\right)\log 2.
\end{multline*}
It remains to estimate the degree of an extension over which $A$ is defined. 

The construction in Lemma~\ref{lem:goodrep} comes from Lemma~\ref{lem:ones}, which involves taking $(d^m-1)$th roots for every $m$-cycle of a function $\sigma$ from the set $\{0, ..., N\}$ to itself, so if there are $t$ periodic points in total, in cycles of length $m_1$, ..., $m_r$, this requires an extension of degree
\[\prod_{i=1}^r(d^{m_i}-1)\leq \prod_{i=1}^rd^{m_i}=d^t.\]
Then, each strictly preperiodic point of $\sigma$ (suppose there are $s$ of these) requires an additional extension of degree at most $d$, showing that $A$ is defined over some extension $L/K$ satisfying
\[[L:K]\leq d^{t+s}\leq d^{N+1}.\]
\end{proof}

\begin{proof}[Proof of Theorem~\ref{th:rigid}]
Let $k$ be an algebraically closed field of characteristic 0, and let $f:V\dashrightarrow\MinCrit_d^N\subseteq\mathsf{M}_d^N$ be a rational map describing a family, the generic fibre of which is PCF. That is, $f$ is a $K=k(V)$ point on $\mathsf{M}_d^N$ corresponding to a minimally critical PCF map.

 After perhaps making a finite extension and a change of variables, we can assume without loss of generality that $f=f_A$ for some $A\in\PGL_{N+1}(K)$, and that $V$ is normal and projective. 

Now, for each irreducible divisor $D$ on some smooth normal model of $V$, we obtain a $D$-adic absolute value on $K$, which is non-archimedean and non-$p$-adic, and these absolute values satisfy the product formula. It follows from Lemma~\ref{lem:lyapbound}, after summing over all places (and possibly changing variables once again) that
\[\left(\frac{d-(N^2+N+1)}{dN}\right)h_{\PGL_{N+1}}(A)\leq \hat{h}_{\mathrm{crit}}(f_A)\leq N(N+1)h_{\PGL_{N+1}}(A).\]
If $f_A$ PCF, then the middle term vanishes, and so $h_{\PGL_{N+1}}(A)=0$, given that $d>N^2+N+1$. But this means that the entries of $A$ are functions on $V$ with no poles, and hence are constant.
\end{proof}

We conclude with an estimate on the difference between the naive and canonical heights of divisors in general for minimally critical maps, which is a more precise special case of a result in~\cite{pnheights}, and which is obtained by summing the inequlities in Lemma~\ref{lem:greens} over all places. Note again that, for $m\in \ZZ^+$, we have $h(m)=\log m$ in the number field case, and $h(m)=0$ in the function field case.
\begin{prop}
Let $f_A$ be as in~\eqref{eq:Aform}, and let $D$ be any effective divisor on $\PP^N$. Then 
\begin{multline*}
-\frac{\deg(D)}{d-1}h_{\mathrm{PGL_{N+1}}(A)}-\left(\frac{\deg(D)}{d-1}+\frac{N}{d^{N+1}-1}\right)h(2)\\ \leq \hat{h}_{f_A}(D)-h(D)\\ \leq \frac{N\deg(D)}{d-1}h_{\mathrm{PGL_{N+1}}(A)}+\frac{\deg(D)}{d-1}h(2N!)+\left(\frac{N\deg(D)}{d^N}+\frac{N}{d^{N+1}-1}\right)h(2).
\end{multline*}
\end{prop}

\end{document}